\documentclass[10pt, a4paper, reqno, english]{amsart}

\usepackage{hyperref}
\usepackage{url}
\usepackage{amssymb,mathrsfs,mathtools}
\usepackage[all]{xy}
\usepackage{enumitem}
\usepackage{tikz}
\usetikzlibrary{matrix,arrows}

\newcommand\ZZ{\mathbb{Z}}

\newcommand\NN{\mathbb{N}}

\newcommand\QQ{\mathbb{Q}}

\newcommand\PP{\mathbb{P}} 

\newcommand{\OO}{\mathcal{O}} 

\DeclareMathOperator{\height}{ht} 
\DeclareMathOperator{\supp}{Supp} 
\DeclareMathOperator{\spec}{Spec} 
\DeclareMathOperator{\pic}{Pic} 
\DeclareMathOperator{\gal}{Gal} 
\DeclareMathOperator{\rk}{rk} 

\newtheorem{theorem}{Theorem}
\newtheorem{lemma}[theorem]{Lemma}
\newtheorem{prop}[theorem]{Proposition}
\newtheorem{cor}[theorem]{Corollary}
\theoremstyle{definition}
\newtheorem{defin}[theorem]{Definition}
\newtheorem{rem}[theorem]{Remark}
\newtheorem{example}[theorem]{Example}

\newtheorem{question}[theorem]{Question}

\numberwithin{theorem}{section}
\numberwithin{equation}{section}


\begin{document}

\setcounter{tocdepth}{1}

\title{On Galois descent of complete intersections}

\address{EPFL SB MATH CAG, B\^at. MA, Station 8, 1015 Lausanne, Switzerland}
  
\author{Marta Pieropan} \email{marta.pieropan@epfl.ch}

\date{March 13, 2020}

\begin{abstract}
We introduce a notion of strict complete intersections with respect to  Cox rings and we prove Galois descent for this new notion.
\end{abstract}

\maketitle

\tableofcontents

\section{Introduction}

Throughout the paper, complete intersection means  scheme-theoretic complete intersection. Given a field $k$ we always denote by $\overline k$ a separable closure. 

Given an inclusion $X\subseteq Y$ of varieties over a field $k$ such that $X_{\overline k}$ is a  complete intersection of hypersurfaces of $Y_{\overline k}$, 
one can ask whether $X$ is a  complete intersection of hypersurfaces of $Y$.
If $Y=\mathbb{P}^n_k$ then the answer is positive. A proof by induction as in \cite[Lemma 3.3]{MR3605019} works over all fields and regardless of the smoothness of the complete intersection. In this paper we generalize the result to polarized log Fano ambient varieties as follows.
\begin{theorem}\label{thm:logFano_ample}
Let $k$ be a field of characteristic 0. 
Let $Y$ be a  log Fano $k$-variety and $X\subseteq Y$ a subvariety such that $X_{\overline k}$ is a complete intersection of $s$ hypersurfaces of $Y_{\overline k}$ of degrees $D_1,\dots,D_s\in\ZZ A$ for a very ample $A\in\pic(Y)$. Then $X$ is a complete intersection of $s$ hypersurfaces of $Y$ of degrees $D_1,\dots,D_s$, respectively.
\end{theorem}

We investigate also complete intersections of hypersurfaces whose degrees are not all multiples of the same divisor class. In this case,
if the divisor class of one of the hypersurfaces that define $X_{\overline k}$ is not defined over $k$,  we should not expect a positive answer to the  question above, as the following example illustrates. 

\begin{example}\label{eg:orbit}
Let $k'/k$ be a quadratic separable extension of fields. Let $\sigma\in\gal(k'/k)$ be the nontrivial element. Let $Y$ be a $k$-variety such that $Y_{k'}\cong\PP^1_{k'}\times\PP^1_{k'}$ with the  $\gal(k'/k)$-action that sends a $k'$-point $((x_1:y_1),(x_2:y_2))$ to the point $((\sigma(x_2):\sigma(y_2)),(\sigma(x_1):\sigma(y_1)))$. Let $H_i:=\{x_i=0\}\subseteq Y_{k'}$ for $i\in\{1,2\}$. Then $X:=H_1\cap H_2$ is a complete intersection defined over $k$, 
as the hypersurfaces $H_1$ and $H_2$ form an orbit under the the  $\gal(k'/k)$-action on $Y_{k'}$.
We observe that $H_1$, $H_2$ have classes $(1,0)$, $(0,1)$ in $\pic(Y_{k'})\cong\ZZ^2$, respectively, and that $\pic(Y)$ is the subgroup generated by the class $(1,1)$. Hence, by intersection theory, $X$
 cannot be written as a complete intersection of hypersurfaces of $Y$. 
\end{example}
To study complete intersections as in the example above we introduce the notion of orbit complete intersection: We say that a subvariety $X$ of an integral $k$-variety  $Y$ is a \emph{single-orbit complete intersection} if $X_{\overline k}$ is a complete intersection of Cartier divisors $H_1,\dots,H_s$ on $Y_{\overline k}$ such that $H_1,\dots,H_s$ form an orbit under the action of $\gal(\overline k/k)$ on $Y_{\overline k}$ and the subset $\{[H_1],\dots,[H_s]\}$ of $\pic(Y_{\overline k})$ has cardinality $s$. We say that $X$ is an \emph{orbit complete intersection} if it is a scheme-theoretic complete intersection of single-orbit complete intersections. 
We observe that  all single-orbit complete intersections in $\PP^n_k$ are hypersurfaces, the intersection $H_1\cap H_2$  in Example \ref{eg:orbit} is a single-orbit complete intersection, and all complete intersections of hypersurfaces of $Y$ are trivially orbit complete intersections. 
In Section \ref{section:strict_ci} we address the following refinement of the original question.

\begin{question}\label{question_refined}
Given an inclusion $X\subseteq Y$ of varieties over a field $k$ such that $X_{\overline k}$ is a  complete intersection of hypersurfaces of $Y_{\overline k}$, is $X$ an orbit complete intersection?
\end{question}

The notion of orbit complete intersection has natural arithmetic applications. Indeed, many proofs involving number theory are carried out assuming that the varieties under study are defined by equations over the base field $k$. Similarly, knowing the action of the absolute Galois group on the equations that define the variety gives control on the arithmetic properties. The notion was inspired by the results in \cite[\S 5.2]{ratconn}. We expect many further arithmetic applications.

We recall that a scheme-theoretic complete intersection $X$ of codimension $s$ in a projective space $\PP^n_{\overline k}$  has the property that the homogeneous ideal of $X$ in the coordinate ring of $\PP^n_{\overline k}$ is generated by $s$ elements. This last property is called strict complete intersection in  \cite[Exercise II.8.4]{MR0282977}.
To study Galois descent of complete intersections we introduce a notion of ideal of a subvariety $X\subseteq Y$ in a Cox ring $R$ of $Y$ and we say that $X$ is a strict complete intersections with respect to the given Cox ring if the ideal of $X$ in $R$ is generated by $s$ elements where $s$ is the codimension of $X$ in $Y$. See Section \ref{section:strict_ci} for the precise definitions. The reason for this definition is that the ideal of $X_{\overline k}$ in $R_{\overline k}$ is $\gal(\overline k/k)$-invariant, and hence can be used to perform Galois descent.
See Theorem \ref{thm:form_ci}, which gives a positive answer to Question \ref{question_refined} in the case of strict complete intersections with respect to Cox rings.

Strict complete intersections with respect to a Cox ring $R$ are, in particular, complete intersections of hypersurfaces defined by elements $f_1,\dots,f_s$ of $R$.
In Section \ref{section:saturation} we show that the strict complete intersection property is equivalent to the saturation of the ideal generated by $f_1,\dots,f_s$ with respect to the irrelevant ideal of the Cox ring (see Corollary \ref{cor:strict_ci_saturation}). 
If $Y$ is a projective space, the saturation is automatic. 
In other ambient varieties not all complete intersections are strict complete intersections with respect to a Cox ring. In Section \ref{section:applications} we give some examples. 

\subsection*{Notation}
Unless  stated otherwise,  $k$ denotes an arbitrary field.
We denote by $\OO_Y$ the structure sheaf of a variety $Y$. The ideal sheaf of a closed subvariety $X\subseteq Y$ is denoted by $\mathscr{I}_{X}$. 
Given an effective Cartier divisor $D$ on an integral variety $Y$, we denote by $\supp(D)$ the support of $D$, and we identify
$H^0(Y,\OO_Y(D))$ with the set of elements $a$ in the function field $K(Y)$ of $Y$ such that $D+(a)$ is effective, where $(a)$ is the principal ideal defined by $a$. 
Given an element $g$ of a ring $R$, we denote by $R[g^{-1}]$ the localization of $R$ at $g$.
Given two ideals $I$ and $G$ in a ring $R$, we denote by $(I:G^\infty)$ the saturation of $I$ with respect to $G$ in $R$.

\section{Galois descent of strict complete intersections}\label{section:strict_ci}
\begin{defin}
Let $Y$ be an integral $k$-variety. 
We say that a closed subvariety $X\subseteq Y$ is a \emph{complete intersection} of hypersurfaces $H_1,\dots,H_s$ of $Y$ if $\dim X+s=\dim Y$ and $
\mathscr{I}_X=\mathscr{I}_{H_1}+\dots+\mathscr{I}_{H_s}$.
\end{defin}

We introduce a correspondence between  ideal sheaves (of subvarieties) and homogeneous ideals in Cox rings.
We refer to \cite{arXiv:1408.5358} for the theory of Cox rings.
\begin{defin}\label{def:indeals_cox_ring}
Let $Y$ be an integral $k$-variety such that $H^0(Y,\OO_{Y_{\overline k}})^\times=\overline k^\times$.
Let $R$ be a Cox ring of $Y$ of type $M\subseteq \pic(Y)$ for a finitely generated subgroup $M$ of $\pic(Y)$. 
For every homogeneous  element $f\in R$, denote by $D_f$ the corresponding effective divisor on $Y$.
For every homogeneous ideal $I$ of $R$ let 
\[
\varphi(I):=\sum_{\substack{f\in I\\\text{homogeneous}}}\OO_Y(-D_f)\subseteq\OO_Y.
\]
For every ideal sheaf $\mathscr{I}\subseteq\OO_Y$, let $\psi(\mathscr{I})$ be the ideal of $R$ generated by all homogeneous elements $f\in R$ such that $\OO_Y(-D_f)\subseteq\mathscr{I}$. In particular, if $X\subseteq Y$ is a subvariety with ideal sheaf $\mathscr{I}_X\subseteq\OO_Y$, we say that $\psi(\mathscr{I}_X)$ is the ideal of $X$ in the Cox ring $R$.
\end{defin}

\begin{defin}
Let $Y$ be an integral $k$-variety such that $H^0(Y,\OO_{Y_{\overline k}})^\times=\overline k^\times$.
Let $R$ be a Cox ring of $Y$ of type $M\subseteq \pic(Y)$ for a finitely generated subgroup $M$ of $\pic(Y)$. 
 We say that a closed subvariety $X\subseteq Y$ of codimension $s$ is a \emph{strict complete intersection with respect to $R$} of hypersurfaces $H_1,\dots,H_s$ defined by $f_1,\dots,f_s\in R$ if  $X$ is the complete intersection of $H_1,\dots,H_s$ and the ideal $\psi(\mathscr{I}_X)$ is generated by  $f_1,\dots,f_s$. 
\end{defin}

The following technical theorem provides a positive answer to Question \ref{question_refined} in the case of strict complete intersections.

 \begin{theorem}\label{thm:form_ci}
Let $k$ be a field and $\overline k$ a separable closure of $k$. Let $Y$ be a  geometrically integral $k$-variety with $H^0(Y_{\overline k},\OO_{Y_{\overline k}})=\overline k$.
Assume that $Y$ admits a Cox ring $R$ over $k$ of type $M\subseteq\pic(Y_{\overline k})$ for a  finitely generated $\gal(\overline k/k)$-invariant subgroup $M$ of $\pic(Y_{\overline k})$.
Let $X\subseteq Y$ be a closed subvariety such that $X_{\overline k}\subseteq Y_{\overline k}$ is a strict complete intersection with respect to $R_{\overline k}$ of hypersurfaces $H_1,\dots,H_s$ defined by $f_1,\dots,f_s\in R_{\overline k}$, respectively.
Assume that there are integers $0=s_0< s_1<\dots< s_n=s$, such that for every $i\in\{0,\dots,n-1\}$ the set
$\{[H_{s_i+1}],\dots,[H_{s_{i+1}}]\}$ forms an orbit of cardinality $s_{i+1}-s_i$ under the $\gal(\overline k/k)$-action on $\pic(Y_{\overline k})$.
Then $X_{\overline k}\subseteq Y_{\overline k}$ is a complete intersection of hypersurfaces $H_1',\dots,H_s'$ such that $H'_{s_i+1},\dots,H'_{s_{i+1}}$ form an orbit under the $\gal(\overline k/k)$-action on $Y_{\overline k}$ for all $i\in\{0,\dots,n-1\}$, and $[H'_i]=[H_i]$ in $\pic(Y_{\overline k})$ for all $i\in\{1,\dots,s\}$.
\end{theorem}
\begin{proof}
Let $k'/k$ be a finite Galois extension such that $f_1,\dots,f_s
\in R_{k'}$ and  $M\subseteq\pic(Y_{k'})$. 
Since $X_{\overline k}$ is a strict complete intersection with respect to $R_{\overline k}$, the ideal $I:=\sum_{i=1}^sf_iR_{k'}$ is invariant under the $\gal(k'/k)$-action on $R_{k'}$.
 For every $i\in\{1,\dots,s\}$, let $S_i\subseteq\gal(k'/k)$ be the stabilizer of $[H_i]$ for the action of $\gal(k'/k)$ on the set $\{[H_1],\dots, [H_s]\}$.

We first prove that 
$X_{\overline k}\subseteq Y_{\overline k}$ is a complete intersection of hypersurfaces $\tilde H_1,\dots,\tilde H_s$ such that $\tilde H_i$ is $S_i$-invariant for all $i\in\{1,\dots,s\}$, and $[\tilde H_i]=[H_i]$ in $\pic(Y_{\overline k})$ for all $i\in\{1,\dots,s\}$.
Let $t\in\{1,\dots,s+1\}$ be the largest integer such that $f_i$ is $S_i$-invariant for all $i\in\{1,\dots,t-1\}$. For all $i\in\{1,\dots,t-1\}$, let $\tilde H_i=H_i$. If $t=s+1$, there is nothing to prove.
If $t\leq s$, let $I_t:=\sum_{\substack{1\leq i\leq s\\i\neq t}}f_iR_{k'}$. Let $V:=\sum_{g\in S_t}g(f_{t})k'\subseteq R_{k'}$.
Then $V$ is $S_t$-invariant, $V\subseteq I$, and all elements of $V$ are homogeneous elements of $R_{k'}$ of degree $[H_t]$.
We denote by $V^{S_t}\subseteq V$ the subset of  $S_t$-invariant elements. Then
$V^{S_t}\nsubseteq I_t$ as $f_t\notin I_t$.
Let $\widetilde f_t\in V^{S_t}\smallsetminus I_t$. Since $\widetilde f_t\in I$, we can write $\widetilde f_t=af_t+b$ with $a\in R_{k'}\smallsetminus\{0\}$ and $b\in I_t$. Since $\deg \widetilde f_t=\deg f_t=[H_t]$ and $I_t$ is a homogeneous ideal, we can assume that $\deg b=[H_t]$ and  $a\in H^0(Y_{k'},\OO_{Y_{k'}})=k'$. Thus $I_t+\widetilde f_tR_{k'}=I$.
Let $\tilde H_t$ be the hypersurface defined by $\widetilde f_t$. Replace $f_t$ by $\widetilde f_t$, $t$ by $t+1$ and repeat the argument. In a finite number of steps we reach the case $t=s+1$.

Given $L_1,L_2\in\pic(Y_{k'})$ we say that $L_1\leq L_2$ if $L_2-L_1$ is an effective divisor class. Then $(\pic(Y_{k'}),\leq)$ is a partially ordered set. We observe that  if $g\in\gal(k'/k)$ and $L_1,L_2\in\pic(Y_{k'})$ satisfy $L_1\leq L_2$, then $gL_1\leq g L_2$. Moreover,
if $g\in\gal(k'/k)$ and $L\in\pic(Y_{k'})$, then $L\leq gL$ is equivalent to $L=gL$ because $g$ has finite order. 
Up to reordering $H_1,\dots,H_s$, we can assume that there are 
$s_0,\dots,s_n\in\{1,\dots,s\}$ as in the statement and $r_1,\dots,r_{m-1}\in\{s_1,\dots,s_{n-1}\}$ with $r_1<\dots<r_{m-1}$, $r_0:=0$, $r_m:=s$ such that 
\begin{enumerate}[ref=(\arabic*)]
\item $[H_i]$ belongs to the $\gal(k'/k)$-orbit of  $[H_j]$ in $\pic(Y_{k'})$ if and only if $ i, j\in\{ r_{l-1}+1,\dots,r_{l}\}$ for some $l\in\{1,\dots,m\}$,
\item $[H_{s_{i}+j}]=[H_{s_{i+1}+j}]$ for all $j\in\{1,\dots,s_{i+1}-s_{i}\}$ for all $i\in\{1,\dots,n\}$ such that $r_{l-1}\leq s_{i},s_{i+2}\leq r_l$ for some $l\in\{1,\dots,m\}$.
\label{assp2}
\end{enumerate}

We conclude the proof by recursion as follows. Let $\alpha\in\{1,\dots,m+1\}$ be the largest number such that $\tilde H_{s_{i-1}+1},\dots,\tilde H_{s_{i}}$ form an orbit under the $\gal(\overline k/k)$-action on $Y_{\overline k}$ for all $i\in\{1,\dots,n\}$ such that $s_i\leq r_{\alpha-1}$. 
For all $i\in\{1,\dots,r_{\alpha-1}\}$, let $H'_i:=\tilde H_i$.
If $\alpha=m+1$, there is nothing to prove. If $\alpha\leq m$, write $I=\bigoplus_{L\in M}I_{L}$ as graded ideal of $R_{k'}$.
Let $\beta:=\#\{[H_{r_{\alpha-1}+1}],\dots,[H_{r_{\alpha}}]\}=s_i-s_{i-1}$ for $i\in\{1,\dots,n\}$ such that $s_i=r_\alpha$. Then $r_{\alpha}-r_{\alpha-1}=\beta\gamma$ for some $\gamma\in\ZZ_{>0}$. 
 For $i\in\{1,\dots,\beta\}$, let $L_i:=[H_{r_{\alpha-1}+i}]$.
For every $i\in\{1,\dots,\beta\}$, the set $\{\tilde f_1,\dots,\tilde f_s\}\cap I_{L_i}$ has cardinality $\gamma$. 
We denote by $f_{i,1},\dots,f_{i,\gamma}$ its elements. 
Let $\delta\in\ZZ_{\geq0}$ such that the $k'$-vector space $I_{L_1}$ has dimension $\gamma+\delta$. Then $\dim_{k'}I_{L_i}=\gamma+\delta$ for all $i\in\{1,\dots,\beta\}$ because they are conjugate to $I_{L_1}$ under the $\gal(k'/k)$-action on $I$. 
For every $i\in\{1,\dots,\beta\}$, choose $h_{i,1},\dots,h_{i,\delta}\in I_{L_i}\cap(\sum_{L<L_i}I_LR_{k'})$ such that $f_{i,1},\dots,f_{i,\gamma},h_{i,1},\dots,h_{i,\delta}$ is a basis of the $k'$-vector space $ I_{L_i}$. Then $h_{i,1},\dots,h_{i,\delta}$ is a basis of the $k'$-vector space $ I_{L_i}\cap(\sum_{L<L_i}I_LR_{k'})$ because $X_{\overline k}$ is a complete intersection. 
For every $i\in\{1,\dots,\beta\}$, let $g_i\in\gal(k'/k)$ such that $g_iL_1=L_i$.
We observe that $g_jg_i^{-1}\sum_{L<L_i}I_LR_{k'}=\sum_{L<L_j}I_LR_{k'}$, because the ideal $I$ is $\gal(k'/k)$-invariant. 
Then $g_i(\sum_{j=1}^\delta h_{1,j}k')=\sum_{j=1}^\delta h_{i,j}k'$ for all $i\in\{1,\dots,\beta\}$. 
Then 
$g_i(f_{1,1}),\dots,g_i(f_{1,\gamma}),h_{i,1},\dots,h_{i,\delta}$ is a basis of $g_iI_{L_1}=I_{L_i}$.
Thus
\[
I=\left(\sum_{i=1}^{r_{\alpha-1}}f_iR_{k'}\right)+\left(\sum_{i=1}^\beta\sum_{j=1}^\gamma g_i(f_{1,j})R_{k'}\right)+\left(\sum_{i=r_\alpha+1}^{s}f_iR_{k'}\right).
\]
By the orbit-stabilizer theorem $\beta=\#\gal(k'/k)/\#S$, where $S$ is the stabilizer of $L_1$. Therefore, the set $\{g_1(f_{1,j}),\dots,g_\beta(f_{1,j})\}$ is an orbit under the $\gal(k'/k)$-action on $R_{k'}$ for all $j\in\{1,\dots,\gamma\}$. 
For every 
$i\in\{1,\dots,\beta\}$ and $j\in\{1,\dots,\gamma\}$, let $H'_{r_{\alpha-1}+(j-1)\beta+i}$
 be the hypersurface defined by $g_i(f_{1,j})$. Then $[H'_{r_{\alpha-1}+(j-1)\beta+i}]=L_i=[H_{r_{\alpha-1}+(j-1)\beta+i}]$ by condition \ref{assp2} above. 
 For every 
$i\in\{1,\dots,\beta\}$ and $j\in\{1,\dots,\gamma\}$, replace $f_{i,j}$ by $g_i(f_{1,j})$.
Replace $\alpha$ by $\alpha+1$ and repeat the argument. In a finite number of steps we reach the case $\alpha=m+1$.
\end{proof}

\section{Ideals of subvarieties in Cox rings}\label{section:saturation}
We study the correspondence between  ideal sheaves of subvarieties and homogeneous ideals in Cox rings introduced in Definition \ref{def:indeals_cox_ring} and we reformulate the strict complete intersection property in terms of saturation of the corresponding ideal in the Cox ring.

\begin{prop}
\label{prop:nullstellensatz2}
Let $k$ be a field.
Let $Y$ be an integral $k$-variety such that $H^0(Y,\OO_{Y_{\overline k}})^\times=\overline k^\times$.
Let $R$ be a Cox ring of $Y$ of type $M\subseteq \pic(Y)$ for a finitely generated subgroup $M$ of $\pic(Y)$. 
Then 
\begin{enumerate}[label=(\roman*), ref=(\roman*)]
\item \label{item:nullstellensatz1}
$\varphi(\sum_{f\in F}fR)=\sum_{f\in F}\OO_Y(-D_f)$ for every set $F$ of homogeneous elements of $R$.
\end{enumerate}
Assume, in addition, that there exist finitely many homogeneous elements $g_1,\dots,g_m$ of $R$ such that $Y\smallsetminus\supp (D_{g_1}),\dots, Y\smallsetminus\supp (D_{g_m})$ form an open covering of $Y$ that refines an affine open covering of $Y$. Then
\begin{enumerate}[label=(\roman*), ref=(\roman*), resume]
\item \label{item:nullstellensatz2}
$\psi(\varphi(I))=(I:(\sum_{i=1}^mg_iR)^\infty)$ for every homogeneous ideal $I$ of $R$.
\item \label{item:nullstellensatz4}
If $M=\pic(Y)$, then $\varphi(\psi(\mathscr{I}))=\mathscr{I}$ for every ideal sheaf $\mathscr{I}\subseteq\OO_Y$.
\end{enumerate}
\end{prop} 
\begin{proof}
To prove \ref{item:nullstellensatz1}, let $F$ be a set of homogeneous elements of $R$. We observe that the inclusion $\sum_{f\in F}\OO_Y(-D_f)\subseteq \varphi(\sum_{f\in F}fR)$ holds by definition. For the reverse inclusion, let $f'\in \sum_{f\in F}fR$ be a homogeneous element.
We have to show that $\OO_Y(-D_{f'})\subseteq \sum_{f\in F}\OO_Y(-D_f)$.  There are $f_1,\dots, f_r\in F$ and $h_1,\dots,h_r\in R$ with $\deg h_i=\deg f'-\deg f_i$ for all $i\in\{1,\dots,r\}$ such that $f'=\sum_{i=1}^rf_ih_i$. Let $R_{[D_{f'}]}$ be the degree-$[D_{f'}]$-part of $R$, and fix an isomorphism $\alpha:R_{[D_{f'}]}\to H^0(Y,\OO_Y(D_{f'}))$ such that $\alpha(f')=1\in K(Y)$. 
Let $\{U_j\}_{j\in J}$ be an affine open covering of $Y$ that trivializes $D_{f'}$, say $D_{f'}=\{(U_j,\alpha_j)\}_{j\in J}$. Then $D_{f_ih_i}=D_{f'}+(\alpha(f_ih_i))=\{(U_j,\alpha_j\alpha(f_ih_i)\}_{j\in J}$ for all $i\in\{1,\dots,r\}$.
Then for every $j\in J$ we have
\[
\OO_Y(-D_{f'})(U_j)=\alpha_j\OO_Y(U_j)\subseteq\sum_{i=1}^r\alpha_j\alpha(f_ih_i)\OO_Y(U_j)=\left(\sum_{i=1}^r\OO_Y(-D_{f_ih_i})\right)(U_j),
\]
as $1=\alpha(f')=\sum_{i=1}^r\alpha(f_ih_i)$. Hence, $\OO_Y(-D_{f'})\subseteq \sum_{i=1}^r\OO_Y(-D_{f_ih_i})$.
Since $D_{f_ih_i}=D_{f_i}+D_{h_i}$, we have $\OO_Y(-D_{f_ih_i})\subseteq \OO_Y(-D_{f_i})$ for all $i\in\{1,\dots, r\}$. 

Now we  prove \ref{item:nullstellensatz2}. 
Let $G:=\sum_{i=1}^mg_iR$. 
To show that $\psi(\varphi(I))\subseteq (I:G^\infty)$,
let $f\in \psi(\varphi(I))$ be a homogeneous element.
Let $\pi:\widehat Y\to Y$ be a torsor  associated to $R$ as in \cite[Theorem 1.1]{arXiv:1408.5358}. 
Then 
\begin{equation}\label{eq:sheaves}
f\OO_{\widehat Y}=\pi^*\OO_Y(-D_f)\subseteq\pi^*\varphi(I)=\sum_{\substack{f'\in I\\\text{homogeneous}}}\pi^*\OO_Y(-D_{f'})=\sum_{\substack{f'\in I\\\text{homogeneous}}}f'\OO_{\widehat Y}.
\end{equation}
Since $\pi$ is affine and $D_{g_i}$ are Cartier divisors, $\pi^{-1}(Y\smallsetminus\supp(D_{g_i}))=\widehat Y\smallsetminus \supp(\pi^*D_{g_i})$ for all $i\in\{1,\dots,m\}$.
Since $\pi^*D_{g_i}$ is the principal ideal defined by $g_i$ for all $i=1,\dots,m$, and $Y\smallsetminus\supp (D_{g_i})$ is contained in an affine open subset of $Y$, then the open subset $V_i:=\pi^{-1}(Y\smallsetminus\supp(D_{g_i}))$ is affine for all $i=1,\dots,m$. We observe that $\OO_{\widehat Y}(V_i)=R[g_i^{-1}]$ for all $i=1,\dots,m$ by \cite[Lemma II.5.14]{MR0463157}.
By  looking at sections of the sheaves in \eqref{eq:sheaves} over the open subsets $V_i$, we get
\[
f|_{V_i}\in \sum_{\substack{f'\in I\\\text{homogeneous}}}f'R[g_i^{-1}]=I R[g_i^{-1}]
\]
for all $i=1,\dots,m$.
So there exists $n\geq0$ such that $g_i^nf\in I$ for every $i=1,\dots,m$.
Let $N:=nm$. Then for every $\alpha_1,\dots,\alpha_m\in\ZZ_{\geq0}$ such that $\alpha_1+\dots+\alpha_m=N$ we have $f\prod_{i=1}^mg_i^{\alpha_i}\in I$ as there exists at least one index $i\in\{1,\dots,m\}$ such that $\alpha_i\geq n$. Thus $fG^N\subseteq I$, which gives $f\in (I:G^\infty)$.

We now prove the reverse inclusion. Let $f\in (I:G^\infty)$. Since $G$ is finitely generated, there exists a positive integer $N$ such that $fG^N\subseteq I$. Then $\OO_{Y}(-D_{fg_i^N})\subseteq\varphi(I)$.
Let $\{U_j\}_j$ be an affine open covering of $Y$ that trivializes simultaneously $D_f$ and $D_{g_i}$ for all $i\in\{1,\dots,m\}$. Write $D_f=\{(U_j,\alpha_j)\}_j$ and $D_{g_i}=\{(U_j,\beta_{i,j})\}_j$ for all $i\in\{1,\dots,m\}$ with $\alpha_j,\beta_{i,j}\in\OO_Y(U_j)$ for all $i,j$. Then 
\begin{multline*}
\OO_Y(-D_f)((Y\smallsetminus \supp D_{g_i})\cap U_j)=\OO_Y(-D_f)(U_j)[\beta_{i,j}^{-1}]=\alpha_j\OO_Y(U_j)[\beta_{i,j}^{-1}]\\
=\alpha_j\beta_{i,j}^N\OO_Y(U_j)[\beta_{i,j}^{-1}]\subseteq\varphi(I)(U_j)[\beta_{i,j}^{-1}]=\varphi(I)((Y\smallsetminus \supp (D_{g_i}))\cap U_j)
\end{multline*}
 for all $i\in\{1,\dots,m\}$ and all $j$. 
 Hence, $\OO_Y(-D_f)\subseteq\varphi(I)$ and $f\in\psi(\varphi(I))$.

For \ref{item:nullstellensatz4},
the inclusion $\varphi(\psi(\mathscr{I}))\subseteq\mathscr{I}$ holds by definition.
For the reverse inclusion, 
it suffices to prove that $\mathscr{I}(Y\smallsetminus \supp(D_{g_i}))\subseteq \varphi(\psi(\mathscr{I}))(Y\smallsetminus \supp(D_{g_i}))$ for all $i\in\{1,\dots,m\}$.
Fix $i\in\{1,\dots,m\}$, and let $D:=D_{g_i}$ and $U:=Y\smallsetminus \supp(D_{g_i})$.
 Let $s\in\mathscr{I}(U)$, and let $D'$ be the principal divisor on $Y$ defined by $s$. Then $D'\cap U$ is an effective divisor on $U$. 
 Let $\{U_j\}_{j\in J}$ be a finite affine open covering of $Y$ that trivializes $D$. For every $j\in J$, let $\alpha_j\in\OO_Y(U_j)$ be a section that defines the principal divisor $D\cap U_j$. 
 Then for every $j\in J$, $\OO_Y(U\cap U_j)=\OO_Y(U_j)[\alpha_j^{-1}]$ and $\mathscr{I}(U\cap U_j)=\mathscr{I}(U_j)[\alpha_j^{-1}]$. Hence, there exists $n_j\in\NN$ such that $\alpha_j^{n_j}s\in \mathscr{I}(U_j)$. Let $n:=\max_{j\in J}n_j$. Then 
$nD+D'$ is an effective Cartier divisor on $Y$ such that
$\mathcal{O}_{Y}(-(nD+D'))\subseteq\mathscr{I}$ and $s\in\mathcal{O}_{Y}(-(nD+D'))(U)$.
\end{proof}

\begin{rem}
If $Y$ is projective and $M$ contains an ample divisor class $A$, 
elements $g_1,\dots,g_m$ as in the statement of Proposition \ref{prop:nullstellensatz2} exist. For example, one can take a basis of the degree-$mA$-part of $R$ for a positive integer $m$ such that $mA$ is very ample.
\end{rem}
\begin{cor}\label{cor:strict_ci_saturation}
Let $Y$ be an integral $k$-variety such that $H^0(Y,\OO_{Y_{\overline k}})^\times=\overline k^\times$. 
Let $R$ be a Cox ring of $Y$ of type $M\subseteq \pic(Y)$ for a finitely generated subgroup $M$ of $\pic(Y)$. 
Assume that there exist finitely many homogeneous elements $g_1,\dots,g_m\in R$ such that $Y\smallsetminus\supp (D_{g_1}),\dots, Y\smallsetminus\supp (D_{g_m})$ form an open covering of $Y$ that refines an affine open covering of $Y$.
Let $X\subseteq Y$ be a complete intersection of hypersurfaces $D_{f_1},\dots,D_{f_s}$ with $f_1,\dots,f_s\in R$. Then $X$ is a strict complete intersection with respect to $R$ if and only if the ideal $\sum_{i=1}^sf_iR$ is saturated with respect to the ideal $\sum_{i=1}^mg_i R$.
\end{cor}
\begin{proof}
By Proposition \ref{prop:nullstellensatz2} we know that $\psi(\mathscr{I}_X)=\psi(\varphi(\sum_{i=1}^sf_iR))$ is the saturation of $\sum_{i=1}^sf_iR$ with respect to $\sum_{i=1}^mg_iR$.
\end{proof}

\section{Applications and examples}\label{section:applications}
In this section we discuss the saturation condition from Corollary \ref{cor:strict_ci_saturation}  and we prove Theorem \ref{thm:logFano_ample}.
\begin{lemma}\label{lem:saturation}
If $R$ is a Cohen-Macaulay 
ring, $I$ and $G$ are ideals in $R$ such that $\height(G)>\height(I)$ and $I$ is generated by $s=\height(I)$ elements, then $I$ is saturated with respect to $G$ in $R$.
\end{lemma}
\begin{proof}
 Let $I=\bigcap_{i=1}^r\mathfrak{q}_i$ be a minimal primary decomposition of $I$ in $R$. Then $(I:G^\infty)=\bigcap_{i=1}^r(\mathfrak{q}_i:G^\infty)$. If $I$ is not saturated with respect to $G$, then there is $i\in\{1,\dots,r\}$ such that $\mathfrak{q}_i\subsetneq (\mathfrak{q}_i:G^\infty)$, then $fG^N\subseteq \mathfrak{q}_i$ for some $f\in (\mathfrak{q}_i:G^\infty)\smallsetminus\mathfrak{q}_i$ and some $N>0$. Since $\mathfrak{q}_i$ is primary, we deduce that $G\subseteq\sqrt{\mathfrak{q}_i}$, so that $\height(G)\leq\height(\sqrt{\mathfrak{q}_i})=\height(\mathfrak{q}_i)$. 
But $\height (\mathfrak{q}_i)=\height(I)$ because the unmixedness theorem holds for $R$ by \cite[Theorem 32, p.110]{MR575344}. This gives a contradiction. 
\end{proof}

\begin{lemma}\label{lem:height_ci}
Let $Y$ be a geometrically integral normal variety over $k$ such that $H^0(Y,\OO_{Y_{\overline k}})^\times= \overline k^\times$. 
Let $R$ be a Cox ring of $Y$ of type $M\subseteq \pic(Y)$ for a finitely generated subgroup $M$ of $\pic(Y)$. 
Assume that $R$ is finitely generated as a $k$-algebra and contains finitely many homogeneous elements $g_1,\dots,g_m$ such that $Y\smallsetminus\supp (D_{g_1}),\dots, Y\smallsetminus\supp (D_{g_m})$ form an affine open covering of $Y$.
Let $X\subseteq Y$ be a complete intersection of hypersurfaces $D_{f_1},\dots,D_{f_s}$ with $f_1,\dots,f_s\in R$. Then $\height(\sum_{i=1}^sf_i R)=s$.
\end{lemma}
\begin{proof}
The ring $R$ is an integral domain  as in \cite[\S5.1]{MR3307753}. 
Let $I:= \sum_{i=1}^sf_i R$. 
 We compute $\height(I)=\dim R-\dim R/I=s$ by using \cite[Corollary 3, p.92]{MR575344} and the fact that $\dim R$ and $\dim R/I$ are the dimensions of torsors under a torus of rank $\rk(M)$ over $Y$ and $X$, respectively,  by \cite[Proposition 4.1]{arXiv:1408.5358}.
\end{proof}

We recall that a variety $Y$ is log Fano if there exists an effective $\QQ$-divisor $D$ such that $(Y,D)$ is klt and $-(K_Y+D)$ is ample. For the singularities we refer to \cite{MR1492525}.

\begin{proof}[Proof of Theorem \ref{thm:logFano_ample}]
Let $R$ be a Cox ring of $Y$ of type $\ZZ A\subseteq\pic(Y_{\overline k})$. 
The $R$ is a Cohen-Macaulay finitely generated $k$-algebra by \cite[Corollary 5.4]{MR3275656}, \cite[Corollary 5.5]{MR1786505} and \cite[Corollary 3.11]{MR1492525}. 
Let $g_0,\dots,g_n$ be a basis of $H^0(Y,\OO_Y(A))$. Let $f_1,\dots,f_s\in R$ be homogeneous elements such that $X_{\overline k}$ is a complete intersection of the hypersurfaces of $Y_{\overline k}$ defined by $f_1,\dots,f_s$.
Let $I:=\sum_{i=1}^sf_iR$ and $G:=\sum_{i=0}^n g_iR$. 
We observe that $R$ is the normalization of its subring $S=\overline k[g_0,\dots,g_n]$ by \cite[Exercise II.5.14(a)]{MR0463157} and that $\spec R\to\spec S$ is an isomorphism away from the closed subsets defined by $G$ in $\spec R$ and by $G\cap S$ in $\spec S$. 
 Since the morphism is finite and $G\cap S$ is a maximal ideal in $S$, we have $\height (G)=\dim R-\dim (R/G)=\dim S-\dim(S/(G\cap S))=\dim S=\dim R=\dim Y+1>s$.
 Then $I$ is saturated with respect to $G$ in $R$ by  Lemmas \ref{lem:saturation} and  \ref{lem:height_ci}, and we can apply Corollary \ref{cor:strict_ci_saturation} and Theorem \ref{thm:form_ci}.
\end{proof}

In the remainder of the section we focus on complete intersections in products of projective spaces.
\begin{prop}
 Let $k$ a field. Fix $n_1,\dots,n_m\geq1$. Let $f_1,\dots,f_r$ be multihomogeneous elements in $R:= k[x_{i,j}:0\leq j\leq n_i, 1\leq i\leq m]$  that define a complete intersection of codimension $s\leq\min_{1\leq i\leq m}n_i$ in $\PP^{n_1}\times_k\dots\times_k\PP^{n_m}$. Then the ideal $(f_1,\dots,f_s)$  in $R$ is saturated with respect to the irrelevant ideal of $R$.
 \end{prop}
 \begin{proof}
The irrelevant ideal of $R$ is $G=\prod_{i=1}^m(x_{i,0},\dots,x_{i,n_m})$. Then $\height G=1+\min_{1\leq i\leq m}n_i> s$. We conclude by Lemmas \ref{lem:saturation} and  \ref{lem:height_ci}. 
 \end{proof}
 
\begin{rem}\label{rem:prime}
A prime ideal $I$ in a ring $R$ is saturated with respect to every ideal $G\not\subseteq I$. Primality is not an easy condition to check in general. However,  if a Cox ring $R$ is isomorphic to a polynomial ring over a field (e.g. if $Y$ is a toric variety), then
every ideal generated by linear polynomials is a prime ideal. 
\end{rem}

The following example shows that the property being a strict complete intersection depend on the choice of the Cox ring.

\begin{example}\label{example:1}
Consider $\PP^1\times \PP^1$ with coordinates $((x_0: x_1),(y_0:y_1))$. The Cox ring  of identity type (i.e.~with $M=\pic(\PP^1\times\PP^1)$) is $R=k[x_0,x_1,y_0,y_1]$ with irrelevant ideal $G=(x_0y_0,x_0y_1,x_1y_0,x_1y_1)$. 
The ideal generated by $x_0y_0$ and $x_1y_1$ in $R$ is not saturated with respect to $G$, as $x_0x_1\in ((x_0y_0,x_1y_1):G^\infty)\smallsetminus (x_0y_0,x_1y_1)$. But the ideal generated by $x_0y_0$ and $x_1y_1$ in the Cox ring of type $\ZZ (1,1)\subseteq\ZZ^2\cong\pic(\PP^1\times\PP^1)$ is saturated with respect to the irrelevant ideal as in the proof of Theorem \ref{thm:logFano_ample}. 

If we allow nonreduced structure, it is easy to construct complete intersections of hypersurfaces of degrees that do not belong to the subgroup $\ZZ (1,1)\subseteq\pic(\PP^1\times\PP^1)$ that are not strict complete intersections with respect to $R$:
the ideal generated by $x_0y_0^2, x_1^2y_1$ in the Cox ring of identity type $R$ defines a complete intersection in $\PP^1\times\PP^1$ but it is not saturated with respect to the irrelevant ideal $G$ in $R$, as $x_0^2x_1^2\in((x_0y_0^2, x_1^2y_1):G^\infty)\smallsetminus (x_0y_0^2, x_1^2y_1)$.
\end{example}

The following example shows that the property being a strict complete intersection with respect to a given Cox ring can depend on the choice of the hypersurfaces defining the complete intersection.

\begin{example}
With the notation of Example \ref{example:1}, 
the point $((0:1),(0:1))$ with reduced scheme structure 
can be written as complete intersection in two different ways: as intersection of the hypersurfaces $\{x_0=0\}$ and $\{y_0=0\}$, or as intersection of the hypersurfaces $\{x_0=0\}$ and $\{x_1y_0=0\}$. 
The ideal generated by $x_0$ and $y_0$ in $R$ is a prime ideal and hence it is saturated with respect to the irrelevant ideal $G$. The ideal generated by $x_0$ and $x_1y_0$ in $R$ is not saturated with respect to $G$, as its saturation is the ideal generated by $x_0$ and $y_0$.
\end{example}

\subsection*{Acknowledgements}
This work was partially supported by grant ES 60/10-1 of the Deutsche Forschungsgemeinschaft. The author thanks the anonymous referee for the helpful remarks.

\bibliographystyle{alpha}
\bibliography{forms_ci_preprints,bibliography}

\end{document}